\documentclass[submission,copyright,creativecommons]{eptcs}
\providecommand{\event}{QPL 2016}
\providecommand{\authorrunning}{K.~Cho and A.A.~Westerbaan}
\providecommand{\titlerunning}{Duplicable von Neumann Algebras}
\usepackage{breakurl}

\newcommand{\urlalt}[2]{\href{#1}{\urlstyle{rm}\nolinkurl{#2}}}

\SetMathAlphabet{\mathcal}{normal}{OMS}{cmsy}{m}{n}
\SetMathAlphabet{\mathcal}{bold}{OMS}{cmsy}{b}{n}

\usepackage{latexsym}
\usepackage{amsmath}
\usepackage{amssymb}
\usepackage{amsthm}
\usepackage{stmaryrd}
\usepackage{relsize}
\usepackage{mathrsfs}
\usepackage{paralist}
\usepackage{nicefrac}
\usepackage{mathrsfs}
\usepackage{mathtools}
\usepackage{enumitem}

\usepackage[all]{xy}
\SelectTips{cm}{}
\newdir{ >}{{}*!/-5pt/@{>}}
\newdir{ (}{{}*!/-5pt/@{(}}
\newdir^{ (}{{}*!/-5pt/@^{(}}
\newdir_{ (}{{}*!/-5pt/@_{(}}

\usepackage{xspace}
\usepackage{picins}
\usepackage[noadjust]{cite}

\newcounter{main}
\newtheorem{theorem}[main]{Theorem}
\newtheorem{lemma}[main]{Lemma}
\newtheorem{proposition}[main]{Proposition}
\newtheorem{corollary}[main]{Corollary}

\theoremstyle{definition}
\newtheorem{definition}[main]{Definition}

\newtheorem{remark}[main]{Remark}

\newcommand{\linf}{\ell^\infty}
\newcommand{\Set}{\mathbf{Set}}
\newcommand{\CvNAMIU}{\mathbf{CvNA}_\mathrm{MIU}}
\newcommand{\vNAMIU}{\mathbf{vNA}_\mathrm{MIU}}
\newcommand{\vNACPsU}{\mathbf{vNA}_\mathrm{CPsU}}
\newcommand{\DvNAMIU}{\mathbf{DvNA}_\mathrm{MIU}}
\newcommand{\CMIU}{\mathbf{C}^*_\mathrm{MIU}}
\newcommand{\CCMIU}{\mathbf{CC}^*_\mathrm{MIU}}
\newcommand{\oc}{\mathord{!}}
\newcommand{\limp}{\mathbin{\multimap}}
\newcommand{\nsp}{\mathrm{nsp}}
\DeclarePairedDelimiter{\sem}{\llbracket}{\rrbracket}

\newcommand{\qcomp}{\mathcal{F}}
\newcommand{\incfun}{\mathcal{J}}

\newcommand{\C}{\mathbb{C}}
\newcommand{\op}{\mathrm{op}}
\newcommand{\id}{\mathrm{id}}

\DeclarePairedDelimiter{\abs}{\lvert}{\rvert}

\newcommand{\Mon}{\mathrm{Mon}}
\newcommand{\CMon}{\mathrm{CMon}}

\newcommand{\catB}{\mathbf{B}}
\newcommand{\catC}{\mathbf{C}}
\newcommand{\scrA}{\mathscr{A}}

\title{Duplicable von Neumann Algebras}

\author{Kenta Cho and Abraham Westerbaan
\institute{Institute for Computing and Information Sciences\\
  Radboud University, Nijmegen, the Netherlands}
\email{\{K.Cho,awesterb\}@cs.ru.nl}
}

\begin{document}
    
\maketitle

\begin{abstract}
Recently, we have
shown that von Neumann algebras
form a model for
 Selinger and Valiron's quantum lambda calculus.
In this paper,
we explain 
our choice of interpretation of the duplicability operator ``$\oc$''
by studying those von Neumann algebras
that might have served as the interpretation
of duplicable types, 
namely those that carry a (commutative) monoid structure
with respect to the spatial tensor product.
We show that every such monoid
is the (possibly infinite) direct product of the complex numbers,
and that
our interpretation of the $\oc$ operator
of the quantum lambda calculus
is exactly the free (commutative) monoid.
\end{abstract}


\section{Introduction}

The quantum lambda calculus~\cite{SelingerV2009},
introduced by Selinger and Valiron,
is a typed programming language
that contains not only a qubit type,
but also a function type $A \limp B$,
and a `duplicable' type $!A$.
It has an (affine) linear type system,
where each variable may be used at most once
unless it has type $\oc A$,
reflecting the no-cloning property of qubits.
In~\cite{CW2016},
we gave a model of the quantum lambda calculus
based on the category~$\vNAMIU$
of von Neumann algebras
and normal unital $*$-homomorphisms (normal MIU-maps, for short).
In this paper,
we  elaborate on the  interpretation
of ``!'' we chose in this model.

Recall that
the duplicability operator ``$\oc$''
is what is called the exponential modality in linear logic.
It is known that a categorical model of the $\oc$ modality
is a so-called \emph{linear exponential comonad} $L$
on a symmetric monoidal category (SMC) $(\catC,\otimes,I)$,
where each object of the form $LA$ is equipped with a commutative comonoid
structure in the SMC $\catC$:
\[
LA\longrightarrow LA\otimes LA
\qquad\qquad
LA\longrightarrow I
\]
The maps corresponds respectively to duplication (contraction rule)
and discarding (weakening rule).
Benton~\cite{Benton1995} later gave a reformulation of a linear exponential comonad as
a \emph{linear-non-linear adjunction}
\[
\xymatrix@C+1pc{
(\catB,\times,1)
\ar@/^1.5ex/[r]
\ar@{}[r]|{\bot}
&
\ar@/^1.5ex/[l]
(\catC,\otimes,I)
},
\]
which induces a linear exponential comonad on $\catC$.
We refer to~\cite{Mellies2009,SelingerV2009} for more details.

To construct a model by von Neumann algebras in~\cite{CW2016},
we used the linear-non-linear adjunction\footnote{%
The opposite category $\vNAMIU^{\op}$ is used here, because
maps between von Neumann algebras are seen as observable/predicate
transformers (Heisenberg's view), which are dual to
state transformers (Schr\"odinger's view).}
\[
\xymatrix@C+1pc{
(\Set,\times,1)
\ar@/^1.5ex/[r]^-{\linf}
\ar@{}[r]|-{\bot}
&
\ar@/^1.5ex/[l]^-{\nsp}
(\vNAMIU^{\op},\otimes,\C).
}
\]
Here~$\nsp(\mathscr{A})=\vNAMIU(\mathscr{A},\C)$
is the set of normal MIU-maps
$\varphi\colon \mathscr{A}\to\mathbb{C}$
for any von Neumann algebra~$\mathscr{A}$,
and~$\linf(X)$
is the von Neumann algebra
of bounded functions $f \colon X\to\mathbb{C}$
for any set~$X$.
Then we interpret
$\oc A$ by $\sem{!A} = \linf(\nsp(\sem{A}))$,
which carries a commutative monoid structure in
the SMC $\vNAMIU$ with
the spatial tensor product $\otimes$.
The interpretation indeed works well in the sense that
our model is adequate with respect to the operational
semantics of the quantum lambda calculus.

However, von Neumann algebras of the form $\linf(X)$
are arguably quite special, excluding
general Abelian von Neumann algebras like $L^\infty([0,1])$.
The natural question arises whether there is a broader class
of von Neumann algebras
that might serve as the interpretation of duplicable types.
This paper gives an answer to this question,
by the definition and theorem below.

\begin{definition}
A von Neumann algebra~$\mathscr{A}$
is \textbf{duplicable}
if there is a \textbf{duplicator} on~$\mathscr{A}$,
that is,
a normal positive subunital linear map
$\delta\colon \mathscr{A}\otimes \mathscr{A}\to\mathscr{A}$
with a \textbf{unit} $u\in \scrA$ such that $0\le u\le 1$,
satisfying
\begin{equation*}
\delta(a\otimes u)\ =\ a\ = \ \delta(u\otimes a)
\quad\text{for all~$a\in\mathscr{A}$.}
\end{equation*}
\end{definition}

\noindent
The unit $u$ can be identified with
a positive subunital map $\tilde{u}\colon \C\to \scrA$ via $\tilde{u}(\lambda)=\lambda u$.
The definition is motivated by the fact that the interpretation of $\oc A$
must carry a commutative monoid structure in $\vNAMIU$.
The condition is weaker, requiring the maps to be only positive subunital,
and dropping associativity and commutativity.
Nevertheless this is sufficient to prove:

\begin{theorem}
\label{thm:duplicable}
A von Neumann algebra~$\mathscr{A}$
is duplicable if and only if
$\mathscr{A}$ is MIU-isomorphic to $\linf(X)$ for some set~$X$.
In that case, the duplicator $(\delta,u)$
is unique, given by
$\delta(a\otimes b) = a\cdot b$ and $u=1$.
\end{theorem}

\noindent
Thus, to interpret duplicable types,
we can really only use von Neumann algebras of the form $\linf(X)$.
It also follows that a von Neumann algebra
is duplicable precisely when it is a (commutative) monoid
in $\vNAMIU$, or in the 
symmetric monoidal category $\vNACPsU$ of von Neumann algebras
and normal completely positive subunital (CPsU) maps.

We further justify our choice,
$\sem{!A} = \linf(\nsp(\sem{A}))$,
by proving that $\linf(\nsp(\mathscr{A}))$
is the free (commutative) monoid on~$\mathscr{A}$ in~$\vNAMIU$.
As a corollary, we also obtain that $\linf(\vNACPsU(\mathscr{A},\mathbb{C}))$
is the free (commutative) monoid on~$\mathscr{A}$
in~$\vNACPsU$.

The paper is organised as follows.
Preliminaries are given in Section~\ref{sec:prelim},
and the proof of Theorem~\ref{thm:duplicable}
in Section~\ref{sec:characterisation-dup-vna}.
We study monoids in the SMCs $\vNAMIU$ and $\vNACPsU$
in Section~\ref{sec:monoids-in-vna},
proving that $\linf(\nsp(\mathscr{A}))$ is the free monoid.
We end the paper
with a discussion of variations
(Dauns' categorical tensor product, $C^*$-algebras),
and related work (on broadcasting and cloning of states)
in Section~\ref{sec:variation-and-related-work}.

\section{Preliminaries}
\label{sec:prelim}

We refer the reader
to the literature
for  details
on von Neumann algebras~\cite{Sakai1998,Takesaki1,kadison1997}
and complete positivity~\cite{paulsen2002}.
Given a von Neumann algebra~$\mathscr{A}$,
we write~$[0,1]_\mathscr{A} = \{a\in \mathscr{A}\mid 0\leq a\leq 1\}$,
and 
we shorten $1-a$ to~$a^\perp$ for all~$a\in [0,1]_\mathscr{A}$.

\subsection{Monoids in monoidal categories}

Let $(\catC,\otimes,I)$ be a symmetric monoidal category (SMC).
A \textbf{monoid} in $\catC$ is an object $A\in\catC$ with
a `multiplication' map
$m\colon A\otimes A\to A$
and a `unit' map $u\colon I\to A$
satisfying the associativity and the unit law,
i.e.\ making the following diagrams commute.
\[
\xymatrix@R-.5pc{
(A\otimes A)\otimes A
\ar[d]_{\alpha}
\ar[rr]^-{m\otimes \id}
&&
A\otimes A
\ar[d]^{m}
\\
A\otimes (A\otimes A)
\ar[r]_-{\id\otimes m}
&
A\otimes A
\ar[r]_-{m}
&
A
}
\qquad
\xymatrix@R-.5pc{
I\otimes A
\ar[dr]_{\lambda}
\ar[r]^-{u\otimes \id}
&
A\otimes A
\ar[d]^{m}
&
\ar[l]_-{\id\otimes u}
A\otimes I
\ar[dl]^{\rho}
\\
&
A
&
}
\]
Here $\alpha,\lambda,\rho$ respectively
denote the associativity isomorphism, and
the left and the right unit isomorphism.
A monoid $(A,m,u)$ is \textbf{commutative} if
$m\circ \gamma=m$,
where $\gamma\colon A\otimes A\to A\otimes A$ is the symmetry isomorphism.
A \textbf{monoid morphism} between monoids $(A_1,m_1,u_1)$
and $(A_2,m_2,u_2)$ is an arrow $f\colon A_1\to A_2$
that satisfies $m_2\circ (f\otimes f)=f \circ m_1$
and $u_2=f \circ u_1$.
We denote the category of monoids
and monoid morphisms in $\catC$
by $\Mon(\catC,\otimes,I)$ or simply by $\Mon(\catC)$.
We write $\CMon(\catC)\subseteq\Mon(\catC)$ for
the full subcategory of commutative monoids.


\subsection{The Model of the Quantum Lambda Calculus by von Neumann algebras}

We quickly review some details
of the authors' previous work~\cite{CW2016}
on a model of the quantum lambda calculus by von Neumann algebras,
because we will need them in this paper.
Let $(\Set,\times,1)$ be the cartesian monoidal category of sets and functions;
$(\vNAMIU,\otimes,\C)$ be the SMC of von Neumann algebras and normal MIU-maps;
and $(\vNACPsU,\otimes,\C)$ be the SMC of von Neumann algebras and normal CPsU-maps
($\otimes$ is the spatial tensor product).
The model is constituted by the following adjunctions.

\begin{equation}
\label{eq:monoidal-adjunctions}
\xymatrix@C=3pc{
(\Set^{\op},\times,1)
\ar@/_1.5ex/[r]_-{\linf}
\ar@{}[r]|-{\bot}
&
\ar@/_1.5ex/[l]_-{\nsp}
(\vNAMIU,\otimes,\C)
\ar@{^{ (}->}@/_1.5ex/[r]_{\incfun}
\ar@{}[r]|{\bot}
&
(\vNACPsU,\otimes,\C)
\ar@/_1.5ex/[l]_{\qcomp}}
\end{equation}

Although we do not use it in the present work,
it is worth mentioning that $\vNAMIU$ is a co-closed SMC~\cite{kornell2012},
which we used (together with the adjunction
$\mathcal{F}\dashv\mathcal{J}$) to interpret function types $A\limp B$
in~\cite{CW2016}.

As is explained in the introduction,
the left-hand side of \eqref{eq:monoidal-adjunctions} models the $\oc$ operator.
Recall that
$\nsp(\scrA)=\vNAMIU(\scrA,\C)$ and
$\linf(X)=\{f\colon X\to\C\mid \sup_x\abs{f(x)}<\infty\}$.
They are indeed functors via
$\nsp(g)(\varphi)=\varphi\circ g$ and $\linf(h)(f)=f\circ h$.
It is straightforward to check that
there is a (dual) adjunction $\nsp\dashv\linf$
via ``swapping arguments''
$g(a)(x)=h(x)(a)$ for normal MIU-maps $g\colon \mathscr{A}\to\ell^\infty(X)$
and functions $h\colon X\to \nsp(\mathscr{A})$.
In~\cite{CW2016} we further observed that:

\begin{lemma}
\label{lem:nsp-linf-ssm}
The functors $\nsp$ and $\linf$ are strong symmetric monoidal,
and there is a symmetric monoidal adjunction $\nsp\dashv\linf$.
Moreover the counit of the adjunction is an isomorphism,
i.e.\ $\nsp(\linf(X))\cong X$.
\qed
\end{lemma}

\begin{corollary}[{\cite[Theorem~IV.3.1]{maclane1998}}]
\label{cor:linf-ff}
The right adjoint functor $\linf$ is full and faithful.
\qed
\end{corollary}

The following is an important observation for the present work.

\begin{corollary}
\label{cor:linfX-cmon}
Von Neumann algebras
of the form $\linf(X)$
carry a commutative monoid structure in $\vNAMIU$.
\end{corollary}
\begin{proof}
Any set $X$ carries a canonical commutative comonoid structure
via the diagonal $X\to X\times X$ and the unique map $X\to 1$,
so that it is a commutative monoid in $(\Set^{\op},\times,1)$.
Because $\linf$ is strong symmetric monoidal,
it preserves the monoid structure as
\[
\linf(X)\otimes\linf(X)\cong
\linf(X\times X)\longrightarrow
\linf(X)
\qquad\qquad
\C\cong \linf(1)\longrightarrow
\linf(X)
\enspace.
\qedhere
\]
\end{proof}

Now we turn to the right-hand side of \eqref{eq:monoidal-adjunctions}.
The functor $\incfun\colon\vNAMIU\hookrightarrow\vNACPsU$
is the inclusion, which is strict symmetric monoidal.
It can be shown \cite{bram2014} via Freyd's Adjoint Functor Theorem that:

\begin{lemma}
\label{lem:adjoint-to-incl}
The inclusion $\incfun\colon\vNAMIU\hookrightarrow\vNACPsU$
has a left adjoint $\qcomp\colon \vNACPsU\to\vNAMIU$.
\qed
\end{lemma}

We remark that the Kleisli category of
the comonad $\qcomp\incfun$ on $\vNAMIU$
is isomorphic to $\vNACPsU$.
Thus the right-hand side of \eqref{eq:monoidal-adjunctions}
is almost a model of Moggi's computational lambda calculi.
That is why our model can accommodate \emph{quantum operations},
which are in general not MIU but CPsU-maps.

\section{Characterisation of Duplicable von Neumann Algebras}
\label{sec:characterisation-dup-vna}

We will prove Theorem~\ref{thm:duplicable} in this section.
Let us give
a rough sketch of the proof.
To show that every duplicable von Neumann algebra
is MIU-isomorphic to~$\linf(X)$ for some set~$X$,
we first prove that~$\mathscr{A}$
is Abelian (in Lemma~\ref{lem:uniqueness-duplicator}).
This reduces the situation
to a measure theoretic problem,
because~$\mathscr{A} \cong L^\infty(X)$
for some appropriate (i.e.~localisable) measure space~$X$.
For simplicity,
we only consider
measure spaces with~$\mu(X)<\infty$
(which are localisable).
This restriction turns out to be harmless
(in the proof of Theorem~\ref{thm:duplicable}).
The measure space~$X$ splits
in a discrete~$D$ and a continuous part~$C$,
so~$X\cong D\oplus C$
(see Lemma~\ref{lem:measure-space-continuous-discrete}).
Since~$L^\infty(D)\cong \linf(D')$
for some set~$D'$,
and $L^\infty(X)\cong \linf(D')\oplus L^\infty(C)$,
our task will be  to show that~$L^\infty(C)=\{0\}$.
Since~$L^\infty(X)$ is duplicable,
we will see that~$L^\infty(C)$ is duplicable as well
(see Corollary~\ref{cor:duplicable-product}).
Thus the crux of the proof is that~$L^\infty(C)$
cannot be duplicable unless~$\mu(C)=0$
(see Lemma~\ref{lem:continuous-measure-space}).

\label{SS:abelian}
\begin{lemma}
\label{lem:unit-duplicator}
Let~$\delta$
be a duplicator 
with unit~$u$
on a von Neumann algebra~$\mathscr{A}$.
Then~$u=1$ and~$\delta(1\otimes 1)=1$.
\end{lemma}
\begin{proof}
Since~$1=\delta(u\otimes 1)\leq \delta(1\otimes 1) \leq 1$,
we have $\delta(u^\perp\otimes 1)=0$.
But then~$u^\perp=0$, and thus~$u=1$,
because  $u^\perp = \delta(u^\perp \otimes u)
\leq \delta(u^\perp \otimes 1) = 0$.
Hence~$u=1$.
Thus~$1=\delta(1\otimes u)=\delta(1\otimes 1)$.
\end{proof}

The following consequence
of Tomiyama's theorem
is based on Lemma~8.3 of~\cite{ndlmcs}.
\begin{lemma}
\label{lem:sef-instrument}
Let~$\mathscr{A}$ be a unital $C^*$-algebra,
and let~$f\colon \mathscr{A}\oplus\mathscr{A}\to \mathscr{A}$
be a positive unital map 
with $f(a,a)=a$ for all~$a\in \mathscr{A}$.
Then $p:=f(1,0)$ is central,
and
for all~$a,b\in\mathscr{A}$
	we have
$f(a,b) \,=\, ap\,+\, bp^\perp$.
\end{lemma}
\begin{proof}
By Tomiyama's theorem we have,
for all~$a,b,c,d\in\mathscr{A}$,
\begin{equation*}
a \,f(c,d)\, b \ = \ f(\,acb\,,\,adb\,).
\end{equation*}
In particular,
for all~$a\in\mathscr{A}$,
we have $ap=af(1,0)=f(a,0)=f(1,0)a=pa$.
Thus~$p$ is central.
Similarly, $f(0,b)=bp^\perp$
for all~$b\in\mathscr{A}$.
Then~$f(a,b)=f(a,0)+f(0,b)=ap+bp^\perp$
for all~$a,b\in\mathscr{A}$.
\end{proof}

\begin{lemma}
\label{lem:uniqueness-duplicator}
Let $\delta\colon\mathscr{A}\otimes \mathscr{A}\to\mathscr{A}$
be a duplicator on a von Neumann algebra~$\mathscr{A}$.
Then~$\mathscr{A}$ is Abelian and~$\delta(a\otimes b)=a\cdot b$
for all~$a,b\in\mathscr{A}$.
\end{lemma}
\begin{proof}
We must show that all~$a\in\mathscr{A}$ are central.
It suffices
to show that all~$p\in [0,1]_\mathscr{A}$ are central
(by the usual reasoning).
Similarly, 
we only need to prove that $\delta(a\otimes p) = a\cdot p$
for all~$a\in\mathscr{A}$ and $p\in [0,1]_\mathscr{A}$.

Let~$p\in[0,1]_\mathscr{A}$ be given.
Define~$f\colon \mathscr{A}\oplus\mathscr{A}\to\mathscr{A}$
by $f(a,b) = \delta(a\otimes p+b\otimes p^\perp)$
for all~$a,b\in\mathscr{A}$.
Then~$f$ is positive, unital,
$f(1,0)=p$,
and 
$f(a,a)=a$
for all~$a\in \mathscr{A}$.
Thus, by Lemma~\ref{lem:sef-instrument},
we get that~$p$ is central,
and  $f(a,b)=ap+bp^\perp$ for all~$a,b\in\mathscr{A}$.
Then~$a\cdot p=f(a,0)=\delta(a \otimes p)$.
\end{proof}
\begin{remark}
The special
case of Lemma~\ref{lem:uniqueness-duplicator}
in which~$\delta$ is \emph{completely} positive
can be found
in the literature,
see for example
Theorem~6 of~\cite{Maassen2010}
(in which~$\mathscr{A}$ is also finite dimensional).
\end{remark}
\begin{corollary}
\label{cor:duplicability-multiplication}
Let~$\mathscr{A}$
be a von Neumann algebra.
Then~$\mathscr{A}$
is duplicable
iff there is
a normal positive linear map $\delta\colon\mathscr{A}\otimes\mathscr{A}
\to \mathscr{A}$
with $\mu(a\otimes b)=a\cdot b$ 
for all~$a,b\in\mathscr{A}$.
\end{corollary}
\begin{corollary}
\label{cor:duplicable-product}
Von Neumann algebras~$\mathscr{A}$ and~$\mathscr{B}$
are duplicable
when  $\mathscr{A}\oplus \mathscr{B}$ is duplicable.
\end{corollary}
\begin{proof}
Let $\delta\colon (\mathscr{A}\oplus\mathscr{B})\otimes
(\mathscr{A}\oplus\mathscr{B})\longrightarrow
\mathscr{A}\oplus\mathscr{B}$
be a duplicator on~$\mathscr{A}\oplus\mathscr{B}$.
By Lemma~\ref{lem:uniqueness-duplicator}
we know that~$\mathscr{A}\oplus\mathscr{B}$
is Abelian,
and that $\delta((a_1,b_1)\otimes (a_2,b_2))
= (a_1a_2,b_1b_2)$
for all $a_1,a_2\in\mathscr{A}$
and~$b_1,b_2\in\mathscr{B}$.

Let~$\kappa_1\colon \mathscr{A}\to\mathscr{A}\oplus\mathscr{B}$
be the normal MIU-map
given by~$\kappa_1(a)=(a,0)$ for all~$a\in\mathscr{A}$.
Let~$\delta_\mathscr{A}$ be the composition of
$\xymatrix@C=3em{
\mathscr{A}\otimes\mathscr{A}
\ar[r]|-{\kappa_1\otimes\kappa_1}
&
(\mathscr{A}\oplus\mathscr{B})
\otimes
(\mathscr{A}\oplus\mathscr{B})
\ar[r]|-{\delta}
&
\mathscr{A}\oplus\mathscr{B}
\ar[r]|-{\pi_1}
&
\mathscr{A}
}$.
Then~$\delta_\mathscr{A}$ is normal, positive,
and
$\delta_\mathscr{A}(a_1\otimes a_2)
\,=\,  \pi_1(\delta((a_1,0)\otimes (a_2,0))) 
\,=\, \pi_1(a_1a_2,0)\,=\,a_1a_2$
for all~$a_1,a_2\in\mathscr{A}$.
Thus, by Corollary~\ref{cor:duplicability-multiplication},
$\mathscr{A}$
is duplicable.
\end{proof}

We will now work towards
the proof that
if~$C$ is a finite measure space,
then
$L^\infty(C)$
cannot be duplicable 
unless~$\mu(C)=0$,
see Lemma~\ref{lem:continuous-finite-measure-space-not-duplicable}.
Let us first fix some terminology
from measure theory (see~\cite{Fremlin2000}).

\begin{definition}
\label{def:measure-space}
Let~$X$ be a measure space.
\begin{enumerate}
\item
The measurable subsets of~$X$
are denoted by~$\Sigma_X$.
\item
A measurable subset~$A$ of~$X$
is \textbf{atomic}
if $0<\mu(A)<\infty$,
and $\mu(A')=\mu(A)$ for all~$A'\in\Sigma_X$
with~$A'\subseteq A$ and~$\mu(A')>0$.

\item
$X$ is \textbf{discrete} if $X$ is covered by atomic measurable subsets.

\item
$X$ is \textbf{continuous}
if~$X$ contains no atomic subsets.
\end{enumerate}
\end{definition}
\begin{definition}
Given a measure space~$X$
with~$\mu(X)<\infty$, 
let
$L^\infty(X)$
denote the von Neumann algebra
of bounded functions $f\colon X\to\mathbb{C}$.
Two such functions $f$ and~$g$ are identified in $L^\infty(X)$
when $f(x)=g(x)$ for almost all~$x\in X$.
Multiplication, addition, involution in~$L^\infty(X)$
are all computer coordinatewise,
and the norm~$\|f\|$ of~$f\in L^\infty(X)$
is the least number $r>0$ such that $|f(x)|\leq r$
for almost all~$x\in X$.
(For more details,
see e.g.~Example~IX/7.2 of~\cite{conway2007}.)
\end{definition}

The following lemma,
which will be very useful,
is a variation on
Zorn's Lemma
that does not require the axiom of choice.
\begin{lemma}
\label{lem:measure-zorn}
Let~$X$ be a measure space with $\mu(X)<\infty$.
Let~$\mathcal{S}$
be a collection of measurable subsets of~$X$
such that~$\bigcup_n A_n\in\mathcal{S}$
for all $A_1\subseteq A_2 \subseteq \dotsb$
in~$\mathcal{S}$.
Then for all~$A\in\mathcal{S}$,
there is~$B\in\mathcal{S}$
with $A\subseteq B$
which is maximal in the sense
that $\mu(B')=\mu(B)$
for all~$B'\in\mathcal{S}$ with $B\subseteq B'$.
\end{lemma}
\begin{proof}
The trick is to consider for every $C\in\Sigma_X$
the quantity
$\beta_C \ = \ \sup\{\,\mu(D)\mid C\subseteq D
\text{ and }D\in \mathcal{S}\,\}$.

Note that $\mu(C)\leq \beta_C \leq \mu(X)$
for all~$C\in\Sigma_X$,
	and $\beta_{C_2}\leq \beta_{C_1}$
	for all~$C_1,C_2\in\Sigma_X$ with $C_1 \subseteq C_2$.
To prove this lemma, it suffices to find~$B\in\Sigma_X$
with $A\subseteq B$ and~$\mu(B)=\beta_B$.

Define~$B_1:= B$.
Pick~$B_2\in\mathcal{S}$
such that $B_1 \subseteq B_2$
and~$\beta_{B_1}-\mu(B_2)\leq \nicefrac{1}{2}$.
Pick~$B_3\in\mathcal{S}$
such that $B_2\subseteq B_3$
and~$\beta_{B_2}-\mu(B_3) \leq \nicefrac{1}{3}$.
Proceeding in this fashion,
we get a sequence $B\equiv B_1\subseteq B_2 \subseteq \dotsb$
in~$\mathcal{S}$
with $\beta_{B_{n}}-\mu(B_{n+1})\leq \nicefrac{1}{n}$
for all~$n\in\mathbb{N}$.
Define~$B:=\bigcup_n B_n$.
Then~$B\in \mathcal{S}$.
Moreover,
\begin{equation*}
\mu(B_1)\,\leq\, \mu(B_2)\,\leq\,
\dotsb \,\leq\,\mu(B)\,\leq\, \beta_B \,\leq\, \dotsb
\,\leq\, \beta_{B_2}\,\leq\, \beta_{B_1}.
\end{equation*}
Since for every~$n\in\mathbb{N}$
we have both $\mu(B_{n+1})\leq \mu(B)\leq \beta_B \leq \beta_{B_n}$
and $\beta_{B_n}- \mu(B_{n+1}) \leq \nicefrac{1}{n}$,
we get $\beta_B-\mu(B)\leq \nicefrac{1}{n}$,
and so~$\beta_B = \mu(B)$.
\end{proof}

\begin{lemma}
\label{lem:measure-space-continuous-discrete}
Let~$X$ be a measure space with~$\mu(X)<\infty$.
Then there is a measurable subset~$D\subseteq X$
such that~$D$ is discrete
and~$X\backslash D$ is continuous.
\end{lemma}
\begin{proof}
Since clearly the countable union
of discrete measurable subsets of~$X$
is again discrete,
there is by Lemma~\ref{lem:measure-zorn}
a discrete measurable subset~$D$ of~$X$
which is maximal in the sense that~$\mu(D')=\mu(D)$
for every discrete measurable subset~$D'$ of~$X$ with $D\subseteq D'$.
To show that~$X\backslash D$ is continuous,
we must prove that~$X\backslash D$
contains no atomic measurable subsets.
If~$A\subseteq X\backslash D$ is an atomic measurable subset
of~$X$,
then~$D\cup A$
is a discrete measurable
subset of~$X$
which contains~$D$,
and $\mu(D\cup A)=\mu(D)\cup \mu(A) > \mu(D)$.
This contradicts the  maximality of~$D$.
Thus~$X\backslash D$ is continuous.
\end{proof}
\begin{lemma}
\label{lem:continuous-measure-space}
Let~$X$ be a continuous measure space
with~$\mu(X)<\infty$.
Then for every~$r\in [0,\mu(X)]$
there is a measurable subset~$A$ of~$X$ with $\mu(X)=r$.
\end{lemma}
\begin{proof}
Let us quickly get rid of the case that~$\mu(X)=0$.
Indeed, then~$r=0$, and so~$A=\varnothing$ will do.
For the remainder, assume that~$\mu(X)>0$.

For starters, we show that for every~$\varepsilon >0$
and~$B\in\Sigma_X$ with~$\mu(B)>0$
there is~$A\in\Sigma_X$ with $A\subseteq B$
and  $0<\mu(A)<\varepsilon$.
Define~$A_1 := B$.
Since~$\mu(B)>0$,
and~$A_1$ is not atomic (because~$X$ is continuous)
there is~$A\in\Sigma_X$ with $A\subseteq A_1$ 
and $\mu(A)\neq \mu(A_1)$.
Since~$\mu(A)+\mu(A_1\backslash A)=\mu(A_1)$,
either $0<\mu(A)\leq \frac{1}{2}\mu(A_1)$
or $0<\mu(X\backslash A)\leq \frac{1}{2}\mu(A_1)$.
In any case,
there is~$A_2\subseteq A_1$
with $A_2\in\Sigma_X$
and $0<\mu(A_2)\leq \frac{1}{2}\mu(A_1)$.
Similarly,
since~$A_2$ is not atomic (because~$X$ is continuous),
there is~$A_3\subseteq A_2$
with~$A_3\in\Sigma_X$ and $0<\mu(A_3)\leq \frac{1}{2}\mu(A_2)$.
Proceeding in a similar fashion,
we obtain a sequence $B\equiv A_1 \supseteq A_2\supseteq \dotsb$
of measurable subsets of~$X$
with $0<\mu(A_n)\leq 2^{-n}\mu(X)$.
Then, for every $\varepsilon >0$
there is~$n\in\mathbb{N}$
such that $0<\mu(A_n)\leq \varepsilon$ and~$A_n\subseteq B$.

Now, 
let us prove that there is~$A\in\Sigma_X$ with $\mu(A)=r$.
By Lemma~\ref{lem:measure-zorn}
there is a measurable
subset~$A$ of~$X$
with $\mu(A)\leq r$
and which is maximal
in the sense that $\mu(A')=\mu(A)$
for all~$A'\in\Sigma_X$
with $\mu(A)\leq r$ and~$A\subseteq A'$.
In fact, we claim that~$\mu(A)=r$.
Indeed, suppose that~$\varepsilon := r-\mu(A)>0$
towards a contradiction.
By the previous discussion,
there is~$C\in\Sigma_X$ with $C\subseteq X\backslash A$
such that $\mu(C)\leq \varepsilon$.
Then~$A\cup C$ is measurable,
and $\mu(A\cup C)=\mu(A)+\mu(C)\leq \mu(A)+\varepsilon\leq r$,
which contradicts the maximality of~$A$.
\end{proof}

\begin{lemma}
\label{lem:continuous-finite-measure-space-not-duplicable}
Let~$X$ be a continuous measure space
with~$\mu(X)<\infty$.
If~$L^\infty(X)$ is duplicable,
then~$\mu(X)=0$.
\end{lemma}
\begin{proof}
Suppose that~$L^\infty(X)$ is duplicable
and~$\mu(X)>0$
towards a contradiction.
Let~$\delta$
be a duplicator
on~$L^\infty(X)$.
By Lemma~\ref{lem:uniqueness-duplicator}
we know that~$\delta(f\otimes g)=f\cdot g$ for all~$f,g\in L^\infty(X)$.

Let~$\omega\colon L^\infty(X)\to \mathbb{C}$
be given by~$\omega(f)=\frac{1}{\mu(X)}\int f \,d\mu$
for all~$f\in L^\infty(X)$.
Then~$\omega$ is normal, positive and unital.
Also, $\omega$ is faithful,
or in other words,
for all~$f\in L^\infty(X)$
with $f\geq 0$ and $\omega(f)=0$
we have $f=0$.
It is known
(see e.g.~Corollary 5.12 of~\cite{Takesaki1})
that  there is a  faithful normal positive unital linear map
$\omega\otimes \omega\colon L^\infty(X)\otimes L^\infty(X)\to \mathbb{C}$
with~$(\omega\otimes \omega)(f\otimes g) = \omega(f)\cdot \omega(g)$
for all~$f,g\in L^\infty(X)$.
We will use
$\omega\otimes \omega$ to tease out a contradiction,
but first we will need a second ingredient.

Since~$X$ is continuous,
we may partition~$X$ into two measurable
subsets of equal measure 
with the aid of Lemma~\ref{lem:continuous-measure-space},
that is,
there are measurable subsets $X_{1}$ and~$X_{2}$
of~$X$ with $X=X_{1}\cup X_{2}$, $X_{1}\cap X_{2}=\varnothing$,
and
$\mu(X_{1})=\mu(X_{2})=\frac{1}{2}\mu(X)$.
Similarly, $X_{1}$ 
can be split into two measurable subsets, $X_{11}$ and $X_{12}$,
of equal measure, and so on.
In this way,
we obtain for every word~$w$ over the alphabet~$\{1,2\}$
--- in symbols, $w\in \{1,2\}^*$ ---
a measurable subset~$X_w$ of~$X$
such that $X_w = X_{w1}\cup X_{w2}$,
$X_{w1}\cap X_{w2}=\varnothing$,
and $\mu(X_{w1})=\mu(X_{w2})=\frac{1}{2}\mu(X_w)$.
It follows that~$\mu(X_w)=\frac{1}{2^{\#w}}\mu(X)$,
where~$\#w$ is the length of the word~$w$.

Now, let~$p_w = \mathbf{1}_{X_w}$ 
be the indicator function of~$X_w$
for every~$w\in \{1,2\}^*$.
Let~$w\in \{1,2\}^*$ be given.
Then~$p_w$ is a projection in~$L^\infty(X)$,
and~$\omega(p_w)=2^{-\#w}$.
Moreover, $p_w = p_{w1}+p_{w2}$,
and so
\begin{alignat*}{3}
p_w\otimes p_w 
\ &=\  
p_{w1}\otimes p_{w1} \,+\,
p_{w1}\otimes p_{w2} \,+\,
p_{w2}\otimes p_{w1} \,+\,
p_{w2}\otimes p_{w2}\\
\ &\geq\ 
p_{w1}\otimes p_{w1} \,+\,
p_{w2}\otimes p_{w2}.
\end{alignat*}
Thus, if we define 
$q_N\ :=\ \sum_{w\in \{1,2\}^N}\,p_w\otimes p_w$
for every natural number~$N$,
where~$\{1,2\}^N$ is the set of words over~$\{1,2\}$ of length~$N$,
then we get a descending sequence $q_1\geq q_2\geq q_3\geq \dotsb$
of projections in~$L^\infty(X)\otimes L^\infty(X)$.
Let~$q$ be the infimum of $q_1\geq q_2 \geq \dotsb$ 
in the set of self-adjoint elements of~$L^\infty(X)\otimes
L^\infty(X)$.
Do we have~$q=0$ ?

On the one hand,
we claim that $\delta(q)=1$, and so~$q\neq 0$.
Indeed,
$\delta(p_w\otimes p_w)=p_w\cdot p_w = p_w$
for all~$w\in \{1,2\}^N$.
Thus $\delta(q_N) = \sum_{w\in \{1,2\}^N}  \delta(p_w\otimes p_w)
= \sum_{w\in\{1,2\}^N} p_w=1$ for all~$N\in \mathbb{N}$.
Hence $\delta(q)=\bigwedge_n \delta(q_N) = 1$,
because~$\delta$ is normal.
On the other hand,
we claim that $(\omega\otimes \omega)(q)=0$,
and so~$q=0$ since~$\omega\otimes \omega$ is 
faithful and $q\geq 0$.
Indeed,
$(\omega\otimes\omega)(q_N)=
\sum_{w\in\{1,2\}^N} \omega(p_w)\cdot\omega(p_w)
= \sum_{w\in\{1,2\}^N} 2^{-N}\cdot 2^{-N} = 2^{-N}$
for all~$N\in \mathbb{N}$,
	and so $(\omega\otimes\omega)(q)
=\bigwedge_N (\omega\otimes\omega)(q_N) = \bigwedge_N 2^{-N}=0$.
Thus, $q=0$ and $q\neq 0$, which is impossible.
\end{proof}

\begin{lemma}
\label{lem:atomic-measure-space}
Let~$A$ be an atomic measure space.
Then~$L^\infty(A)\cong \mathbb{C}$.
\end{lemma}
\begin{proof}
Let~$f\in L^\infty(A)$ be given.
It suffices to show that
there is~$z\in \mathbb{C}$
such that
$f(x)=z$ for almost all~$x\in A$.
Moreover, we only need to consider the case
that~$f$ takes its values in~$\mathbb{R}$
(because we may split~$f$ in its real and imaginary parts,
and in turn split these in positive and negative parts).

Since,
writing  $I_n = (n,n+1]$,
we have $\mu(A) = \sum_{n\in\mathbb{Z}} f^{-1}(I_n)$,
and~$A$ is atomic,
there a (unique)~$n\in \mathbb{N}$
with  $\mu(A)=\mu(f^{-1}(I_n))$.
Similarly,
writing 
$J_{1} = (n,\frac{2n+1}{2}]$
and $J_{2}=(\frac{2n+1}{2},n+1]$,
we have
$\mu(A) = \mu(J_{1}) + \mu(J_{2})$,
and so there is a (unique)  $m\in \{1,2\}$
with $\mu(A)=\mu(f^{-1}(J_{m}))$.

Continuing in this way,
we can find real numbers $s_1 \leq s_2 \leq \dotsb \leq t_2 \leq t_1$
with $t_n-s_n \leq 2^{-n}$
and $\mu(A)=\mu( f^{-1}(\ (s_n,t_n]\ ))$
for all~$n\in \mathbb{N}$.
Of course,
we also have~$\mu(A)=\mu(f^{-1}(\ [s_n,t_n]\ ))$,
and so
\begin{equation*}
\textstyle
\mu(A) \ =\  \bigwedge_n \mu(f^{-1}(\ [s_n,t_n]\ ))
\ =\  \mu(f^{-1}(\ \bigcap_n [s_n,t_n]\ )).
\end{equation*}
Since~$t_n-s_n \to 0$ as~$n\to \infty$,
there is a real number~$\lambda\in \mathbb{R}$
with $\{\lambda \} = \bigcap_n[s_n,t_n]$,
and so~$\mu(A)=f^{-1}(\{\lambda\})$.
Hence~$f(x)=\lambda$ for almost all~$x\in A$.
\end{proof}
\begin{lemma}
\label{lem:measure-space-partition}
Let~$X$ be a measure space with~$\mu(\mathscr{A})<\infty$.
Then for every partition~$\mathcal{A}$
of~$X$
consisting of measurable subsets,
we have~$L^\infty(X)\cong \bigoplus_{A\in\mathcal{A}} L^\infty(A)$.
\end{lemma}
\begin{proof}
Note that since~$\sum_{A\in\mathcal{A}} \mu(A)=\mu(X)$,
the set $\mathcal{A}' = \{A\in \mathcal{A}\mid \mu(A)>0\}$
is countable.
We will also
need the fact that~$A_0 = \bigcup\{ A\in\mathcal{A}\mid \mu(A)=0\}$
is negligible.
To see this,
note that
$X\backslash A_0$ is 
the union of~$\mathcal{A}'$
and thus measurable.
Further, we have $\mu(X\backslash A_0) = 
\sum_{A\in \mathcal{A}'} \mu(A)
= \sum_{A\in\mathcal{A}} \mu(A)=\mu(X)$,
and thus~$\mu(A_0)=0$.

Let~$A\in \mathcal{A}$
and~$f \in L^\infty(X)$ be given.
Then the restriction $f|A\colon A\to \mathbb{C}$
is an element of~$L^\infty(A)$.
It is not hard to see that $f\mapsto f|_A$
gives a MIU-map
$R_A\colon L^\infty(X)\to L^\infty(A)$.
Let~$R\colon L^\infty(X)\to \bigoplus_{A\in \mathcal{A}} L^\infty(A)$
be given by $R(f) = (R_A(f))_{A\in \mathcal{A}}$.
We claim that~$R$ is bijective (and thus a normal MIU-isomorphism).

To show that~$R$ is injective,
let $f\in L^\infty(X)$ with~$R(f)=0$ be given.
Then for every~$A\in \mathcal{A}$
there is a negligible subset~$N_A$ of~$A$ with $f(x)=0$
for all~$x\in A\backslash N_A$.
Thus~$f(x)=0$ for all~$x\in X\backslash\bigcup_{A\in \mathcal{A}} N_A$.
Since~$N := \bigcup_{A\in\mathcal{A}} N_A 
\subseteq A_0 \cup \bigcup_{A\in\mathcal{A}'} N_A$,
we see that~$N$ is negligible,
and so~$f(x)=0$ for almost all~$x\in X$.
Thus $f=0$ in~$L^\infty(X)$,
and thus~$R$ is injective.

To show that~$R$ is surjective,
let $f\in \bigoplus_{A\in \mathcal{A}} L^\infty(A)$
be given,
and define~$g\colon X\to \mathbb{C}$
by $g(x)=f_A(x)$
for all~$A\in\mathcal{A}$ and~$x\in A$.
We claim that~$g\in L^\infty(X)$,
and then clearly $R(g)=f$.
To begin, we show that $g$ is measurable.
Let~$U$ be a measurable subset of~$\mathbb{C}$.
We must show that~$g^{-1}(U)$ is measurable.
We have
\begin{equation*}
\textstyle 
g^{-1}(U)\ = \ \bigcup_{A\in \mathcal{A}} f^{-1}_A(U)
\ = \ \bigcup_{A\in \mathcal{A}'} \, f^{-1}_A(U)
\ \,\cup\ \,\bigcup_{A\in \mathcal{A}\backslash \mathcal{A}'}\, f^{-1}_A(U).
\end{equation*}
Note that~$\bigcup_{A\in\mathcal{A}'}f_A^{-1}(A)$
is measurable
(being a countable union of measurable sets),
and that $\bigcup_{A\in\mathcal{A}\backslash\mathcal{A}'} f_A^{-1}(A)$
is negligible
(being a subset of the negligible set~$A_0$).
Thus~$g^{-1}(U)$ is measurable.
Hence~$g$ is measurable.
It remains to be shown
that $g$ is essentially bounded,
that is, that there is $r>0$ such that $|f(x)|\leq r$
for almost all~$r\in X$.
For every~$A\in\mathcal{A}$
there is a negligible
subset~$N_A$ of~$A$
such that $|f_A(x)|\leq \|f_A\|$
for all~$x\in A\backslash N_A$.
Since~$\|f\|=\sup_{A\in\mathcal{A}}\|f_A\|$,
we have $|g|\leq \|f\|$
for all~$x\in X\backslash (\bigcup_{A\in\mathcal{A}}N_A)$.
Since~$N:=\bigcup_{A\in\mathcal{A}}N_A\,\subseteq\,
A_0\cup\bigcup_{A\in\mathcal{A}'}A$,
we see that~$N$ is negligible,
and thus that~$g$ is essentially bounded.
Of course, $R(g)=f$, and thus~$R$ is surjective.
\end{proof}

\begin{corollary}
\label{cor:discrete-ell-x}
For every discrete 
measure space~$X$ with~$\mu(X)<\infty$
there is a  set~$Y$ with $L^\infty(X)\cong \linf(Y)$.
\end{corollary}

We are now ready to give the proof
the main result of this paper.
\begin{proof}[Proof of Theorem~\ref{thm:duplicable}]
We have already seen that $\linf(X)$
can be  equipped with a commutative monoid
structure in~$\vNAMIU$
for any set~$X$,
and is thus duplicable.
Conversely,
let~$\delta\colon \mathscr{A}\otimes\mathscr{A}\to\mathscr{A}$
be a duplicator with unit~$u$ on a von Neumann algebra~$\mathscr{A}$.
By Lemma~\ref{lem:unit-duplicator}, we know that~$u=1$,
and by Lemma~\ref{lem:uniqueness-duplicator},
we know that~$\mathscr{A}$
is Abelian 
and~$\delta(a\otimes b)=a\cdot b$
for all~$a,b\in \mathscr{A}$.
Thus, the only thing that remains to be shown
is that~$\mathscr{A}$ is MIU-isomorphic to $\linf(Y)$
for some set~$Y$.

It is known that 
any Abelian von Neumann algebra must be of the form $L^\infty(X)$,
where~$X$ is a  measure space.
Moreover, $X$ can be taken to be a \emph{localisable} measure space,
but we will not need the general theory of localisable measure spaces here.
Instead,
we can get away with using the fact
(obtained 
by inspecting the proof of Proposition~1.18.1
of~\cite{Sakai1998}),
that there is a family of measure spaces $(X_i)_{i\in I}$
with $\mathscr{A}\cong \bigoplus_{i\in I} L^\infty(X_i)$
and $\mu(X_i)<\infty$ for all~$i\in I$.
To prove that~$\mathscr{A}\cong \ell^\infty(Y)$
for some set~$Y$ it suffices
to show that there is for every~$i\in I$
 a set~$Y_i$
with $L^\infty(X_i)\cong \ell^\infty(Y_i)$,
because then 
\begin{equation*}
\textstyle \mathscr{A}\ \cong \ 
\bigoplus_{i\in I} \ell^\infty(Y_i)\ \cong\ 
\ell^\infty\bigl(\,\bigcup_{i\in I} Y_i\,\bigr).
\end{equation*}

Let~$i\in I$ be given.
We must find a set~$Y$ 
such that~$L^\infty(X_i)\cong \ell^\infty(Y)$.
Since~$\mathscr{A}\cong L^\infty(X_i)\,\oplus\,\bigoplus_{j\neq i} 
L^\infty(X_j)$ is duplicable,
$L^\infty(X_i)$ is duplicable
by  Corollary~\ref{cor:duplicable-product}.
By Lemma~\ref{lem:measure-space-continuous-discrete},
there is a measurable subset~$D$ of~$X_i$ such that~$D$
is discrete, and $C:=X\backslash D$ is continuous.
We have~$L^\infty(X_i)\cong L^\infty(D)\oplus L^\infty(C)$
by Lemma~\ref{lem:measure-space-partition},
and
so $L^\infty(D)$ and~$L^\infty(C)$
are duplicable
(again by Corollary~\ref{cor:duplicable-product}).

By Lemma~\ref{lem:continuous-finite-measure-space-not-duplicable},
$L^\infty(C)$
can only be duplicable if~$\mu(C)=0$,
and so~$L^\infty(C)\cong \{0\}$.
On the other hand,
since~$D$ is discrete,
we have~$L^\infty(D)\cong \ell^\infty(Y)$
for some set~$Y$
(by Corollary~\ref{cor:discrete-ell-x}).
Thus we have
\begin{equation*}
L^\infty(X_i)
\ \cong\  L^\infty(D)\,\oplus\, L^\infty(C)
\ \cong\  \ell^\infty(Y)\,\oplus\, \{0\}
\ \cong\  \ell^\infty(Y).
\end{equation*}
Hence~$\mathscr{A}\cong \ell^\infty(Z)$
for some set~$Z$.
\end{proof}

\section{Monoids in $\vNAMIU$ and~$\vNACPsU$}
\label{sec:monoids-in-vna}

First we characterise duplicable von Neumann algebras
in terms of monoids.

\begin{proposition}
\label{prop:dup-vna-is-monoid}
Let $\scrA$ be a von Neumann algebra.
The following are equivalent.
\begin{enumerate}
\item
$\scrA$ is duplicable.
\item
$\scrA$ carries a monoid structure in $(\vNAMIU,\otimes,\C)$.
\item
$\scrA$ carries a monoid structure in $(\vNACPsU,\otimes,\C)$.
\end{enumerate}
In that case, $\scrA$ is a commutative monoid
(both in $\vNAMIU$ and $\vNACPsU$),
and the monoid structure $(m\colon\scrA\otimes\scrA\to\scrA,u\colon\C\to\scrA)$
is a duplicator on $\scrA$
(when $u$ is identified with $u(1)\in\scrA$).
By Theorem~\ref{thm:duplicable}, a monoid structure on $\scrA$ is unique.
\end{proposition}
\begin{proof}
(1 $\Rightarrow$ 2)
By Theorem~\ref{thm:duplicable},
$\scrA$ is MIU-isomorphic to $\linf(X)$ for some set $X$.
By Corollary~\ref{cor:linfX-cmon},
$\linf(X)$ carries a commutative monoid structure in $\vNAMIU$.
Thus we can equip $\scrA$ with a commutative monoid structure in $\vNAMIU$
via the isomorphism $\scrA\cong\linf(X)$.

(2 $\Rightarrow$ 3) Trivial.

(3 $\Rightarrow$ 1)
Let $m\colon\scrA\otimes\scrA\to\scrA$ and $u\colon\C\to\scrA$ be a monoid structure
on $\scrA$ in $\vNACPsU$.
Then $m(u(1)\otimes a)=
(m\circ (u\otimes \id))(1\otimes a)=
\lambda(1\otimes a)=1\cdot a=a$.
and similarly $m(a\otimes u(1))=a$.
Thus $\scrA$ is duplicable via $m$ and $u(1)$.
\end{proof}

It follows that there is no distinction
between (commutative) monoids in $\vNAMIU$ and in $\vNACPsU$.

\begin{proposition}
\label{prop:cmon-mon-vNA}
$\CMon(\vNAMIU)
=\Mon(\vNAMIU)
=\CMon(\vNACPsU)
=\Mon(\vNACPsU)$.
\end{proposition}
\begin{proof}
By Proposition~\ref{prop:dup-vna-is-monoid},
$\Mon(\vNAMIU)$ and $\Mon(\vNACPsU)$ have the same objects,
and $\CMon(\vNAMIU)=\Mon(\vNAMIU)$ and $\CMon(\vNACPsU)=\Mon(\vNACPsU)$.
Let $f\colon (\scrA_1,m_1,u_1)\to (\scrA_2,m_2,u_2)$
be a morphism in $\Mon(\vNACPsU)$.
Then $f(1)=(f\circ u_1)(1)=u_2(1)=1$
and $f(ab)=(f\circ m_1)(a\otimes b)=(m_2\circ (f\otimes f))(a\otimes b)=f(a)f(b)$,
using the fact that the monoid structure is a unique duplicator.
Thus $f$ is a normal MIU-map, and
we are done since $\Mon(\vNAMIU)\subseteq \Mon(\vNACPsU)$.
\end{proof}

It turns out that $\linf(\nsp(\mathscr{A}))$,
our interpretation of the $\oc$ operator in the quantum lambda calculus,
is exactly the free (commutative) monoid on $\scrA$ in $\vNAMIU$.

\begin{theorem}
\label{thm:free-monoid-in-vNAMIU}
Let~$\mathscr{A}$
be a von Neumann algebra,
and let~$\eta\colon \mathscr{A}\to\linf(\nsp(\mathscr{A}))$
be the normal MIU-map
given by~$\eta(a)(\varphi)= \varphi(a)$.
Then $\linf(\nsp(\mathscr{A}))$
is the free (commutative) monoid
on~$\mathscr{A}$
in~$\vNAMIU$ via~$\eta$.
\end{theorem}
\begin{proof}
Let~$\mathscr{B}$
be a monoid 
on~$\vNAMIU$,
and let~$f\colon \mathscr{A}\to\mathscr{B}$
be a normal MIU-map
We must show that
there is a unique
monoid morphism
$g\colon \linf(\nsp(\mathscr{A}))
\rightarrow \mathscr{B}$
such that~$g\circ \eta = f$.

By Theorem~\ref{thm:duplicable},
we may assume that~$\mathscr{B}=\linf(Y)$
for some set~$Y$.
Since~$\nsp\colon \vNAMIU^\op\to \Set$
is left adjoint
to~$\linf\colon \Set \to \vNAMIU^\op$
with unit~$\eta$ (Lemma~\ref{lem:nsp-linf-ssm}),
there is a unique map $h\colon Y\to \nsp(\mathscr{A})$
with $\linf(h)\circ \eta = f$.
Since~$\linf$ is full and faithful
by Corollary~\ref{cor:linf-ff},
the only thing that remains to be shown is that~$\linf(h)$
is a monoid morphism.
Indeed it is,
since the monoid multiplication
on~$\linf(\nsp(\mathscr{A}))$
and~$\linf(Y)$
is given by ordinary multiplication,
which is preserved by~$\linf(h)$,
being a MIU-map.
\end{proof}

\begin{corollary}
Let~$\mathscr{A}$
be a von Neumann algebra.
Then $\linf(\vNACPsU(\mathscr{A},\C))$
is the free (commutative) monoid
on~$\mathscr{A}$ in~$\vNACPsU$.
\end{corollary}
\begin{proof}
Theorem~\ref{thm:free-monoid-in-vNAMIU} asserts that
$\linf\circ\nsp$ is a left adjoint to
the forgetful functor $\Mon(\vNAMIU)\to\vNAMIU$.
By Prop.~\ref{prop:cmon-mon-vNA},
the forgetful functor $\Mon(\vNACPsU)\to\vNACPsU$
factors through $\vNAMIU$ as:
\[
\xymatrix{
\Mon(\vNACPsU)
\ar@{=}[r]
&
\Mon(\vNAMIU)
\ar[r]^-{\bot}&
\ar@/_3ex/[l]_{\linf\circ\nsp}
\vNAMIU
\ar@{^{ (}->}[r]^-{\bot}&
\ar@/_3ex/[l]_{\qcomp}
\vNACPsU
}
\]
where $\qcomp$ is the adjoint functor of Lemma~\ref{lem:adjoint-to-incl}.
Thus the free monoid on $\scrA$ in $\vNACPsU$ is given by:
\[
(\linf\circ\nsp\circ\qcomp)(\scrA)
=
\linf(\vNAMIU(\qcomp\scrA,\C))
\cong
\linf(\vNACPsU(\scrA,\C))
\enspace.\qedhere
\]
\end{proof}

Finally we observe that duplicable von Neumann algebras and
monoids in $\vNAMIU$ (or in $\vNACPsU$) are simply identified with sets.
Let $\DvNAMIU\subseteq\vNAMIU$ denote the full subcategory
consisting of duplicable von Neumann algebras.

\begin{proposition}
$\Mon(\vNAMIU)\cong\DvNAMIU\simeq\Set^{\op}$.
\end{proposition}
\begin{proof}
($\Mon(\vNAMIU)\cong\DvNAMIU$)
By Proposition~\ref{prop:dup-vna-is-monoid},
we have the forgetful functor $U\colon\Mon(\vNAMIU)\to\DvNAMIU$
that is bijective on objects.
We need to prove that the functor is full,
namely that any normal MIU-map between monoids
is a monoid morphism.
The monoid structure is preserved by normal MIU-maps
since the monoid structure is a duplicator
by Proposition~\ref{prop:dup-vna-is-monoid},
and the duplicator is given by the multiplication and unit
of the von Neumann algebra by Theorem~\ref{thm:duplicable}.

($\DvNAMIU\simeq\Set^{\op}$)
We have a functor $\linf\colon\Set^{\op}\to\DvNAMIU$,
which is full and faithful by Corollary~\ref{cor:linf-ff},
and also essentially surjective
by Theorem~\ref{thm:duplicable}.
\end{proof}

\begin{remark}
Note that $\linf(\nsp(\mathscr{A}))$
is a free commutative comonoid on $\scrA$ in $\vNAMIU^{\op}$.
Thus $\vNAMIU^{\op}$ gives a model of intuitionistic linear logic
formulated by Lafont (called a \emph{Lafont category}),
see e.g.~\cite{Mellies2009}.
We leave it for future work to check if the explicit construction
of the free commutative comonoid by Melli\`es et al.~\cite{MelliesTT2009} works
for $\vNAMIU^{\op}$.
\end{remark}

\section{Variations and Related Work}
\label{sec:variation-and-related-work}

\subsection{Categorical tensor product}
Dauns~\cite{dauns1972},
and later Kornell~\cite{kornell2012},
have considered
an alternative to the spatial tensor product, 
$\otimes$,
called the categorical tensor product, $\tilde \otimes$.
We bring this up,
because, while only the von Neumann algebras of the form~$\linf(X)$
carry a monoid structure in~$(\vNAMIU,\otimes, \mathbb{C})$,
\emph{all}
Abelian von Neumann algebras carry 
a commutative monoid structure in~$(\vNAMIU,\tilde\otimes,\mathbb{C})$,
and in fact,
the category
of (commutative) monoids on
$(\vNAMIU,\tilde\otimes,\mathbb{C})$,
is equivalent to~$\CvNAMIU$.
We will only sketch a proof of these statements here.

Let us quickly describe~$\tilde\otimes$.
Let~$\mathscr{A}$ and~$\mathscr{B}$
be von Neumann algebras.
Then $\mathscr{A}\tilde\otimes\mathscr{B}$
is a von Neumann algebra
equipped with a map $\tilde\otimes \colon \mathscr{A}\times\mathscr{B}\to
\mathscr{A}\tilde\otimes\mathscr{B}$,
and $\mathscr{A}\tilde\otimes \mathscr{B}$
has the following `universal property'.
For every von Neumann algebra~$\mathscr{C}$,
and for all normal MIU-maps $f\colon \mathscr{C}\to\mathscr{A}$
and $g\colon \mathscr{C}\to\mathscr{B}$
with $f(c)g(d)=g(d)f(c)$ for all~$c,d\in\mathscr{C}$,
there is a unique normal MIU-map $h\colon \mathscr{A}\tilde\otimes\mathscr{B}
\to\mathscr{C}$ with $h(a\tilde\otimes b) =f(a)g(b)$
for all $a,b\in\mathscr{A}$. 
(This universal property differs slightly
from the one given in Definition~6.2 of~\cite{kornell2012}
but nonetheless correct,
as one easily sees by inspecting
 the proof of Proposition~6.1 of~\cite{kornell2012}.)

With the universal property,
one can extend~$\tilde\otimes$
to a bifunctor on~$\vNAMIU$ 
(via $(f_1\tilde \otimes f_2)(a_1\tilde \otimes a_2)
= f_1(a_1)\tilde\otimes f_2(a_2)$),
and easily sees that
$(\vNAMIU,\tilde\otimes,\mathbb{C})$
is  a monoidal category
(cf.~Theorem~III of~\cite{dauns1972}).

Let~$\mathscr{A}$
be an Abelian von Neumann algebra.
Since~$a b = ba$
for all~$a,b\in\mathscr{A}$,
there is a unique normal MIU-map
$\delta\colon \mathscr{A}\tilde\otimes\mathscr{A}
\to \mathscr{A}$
with~$\delta(a\tilde\otimes b)= a b$.
Then~$(\mathscr{A},\delta,{!})$
is a monoid in~$(\vNAMIU,\tilde\otimes,\mathbb{C})$,
where~$!\colon \mathbb{C}\to\mathscr{A}$
is the unique normal MIU-map of this type.
Conversely,
let~$(\mathscr{A},m,u)$
be a monoid 
in $(\vNAMIU,\tilde\otimes,\mathbb{C})$.
Then  clearly~$u={!}$.
Further,
by replacing~$\otimes$ by~$\tilde\otimes$
in the proof of Lemma~\ref{lem:uniqueness-duplicator},
we see that~$\mathscr{A}$ is Abelian and~$m(a\tilde\otimes b)=ab$.
In particular,
any normal MIU-map between
monoids in $(\vNAMIU,\tilde\otimes,\mathbb{C})$
is a monoid morphism.
Hence
the category of (commutative) monoids
on $(\vNAMIU,\tilde\otimes,\mathbb{C})$
is equivalent to~$\CvNAMIU$.

One may wonder 
if there is a free monoid in $(\vNAMIU,\tilde\otimes,\mathbb{C})$
on a von Neumann algebra~$\mathscr{A}$.
This boils down to finding a left adjoint to
the inclusion~$\mathcal{J}\colon \CvNAMIU\to\vNAMIU$.
Let us just mention
that there is a left adjoint
to~$\mathcal{J}$,
which maps a von Neumann algebra~$\mathscr{A}$
to $\{a\in \mathscr{A}\mid
\forall b,c\in\mathscr{A}\ [\ abc=cab\ ]\}$.

Recall that the spatial tensor product~$\otimes$
is used to built our model of the quantum lambda calculus.
We would like to mention that
$\tilde\otimes$ cannot serve
the same role,
because~$(\vNAMIU,\tilde\otimes,\mathbb{C})$
is not monoidal closed
(see Corollary~6.5 of~\cite{kornell2012}),
and we do not know
if~$\tilde\otimes$ extends to a functor on~$\vNACPsU$.

\subsection{$C^*$-algebras}
Let~$\CMIU$ be
the category of unital $C^*$-algebras
and MIU-maps.
Then~$(\CMIU,\otimes,\mathbb{C})$
is an SMC
where~$\otimes$ is the spatial tensor product
(see Proposition~2.7 of~\cite{cho2014}).
In contrast
with~$\vNAMIU$, \emph{all}
commutative $C^*$-algebras
carry a commutative monoid structure
in~$(\CMIU,\otimes,\mathbb{C})$,
the category of (commutative)
monoids on~$(\CMIU,\otimes,\mathbb{C})$
is equivalent to
the category~$\CCMIU$
of commutative $C^*$-algebras,
and the free (commutative) monoid
on~$(\CMIU,\otimes,\mathbb{C})$
on a $C^*$-algebra~$\mathscr{A}$
is $C(\mathrm{sp}(\mathscr{A}))$,
where~$\mathrm{sp}(\mathscr{A})$
is the spectrum of~$\mathscr{A}$
(i.e.~the compact Hausdorff space
of MIU-maps from~$\mathscr{A}\to\mathbb{C}$).
We only have space for some hints here.

To see that every
commutative $C^*$-algebra~$\mathscr{A}$
carries a
commutative monoid structure,
note that the spatial tensor product $\mathscr{A}\otimes \mathscr{A}$
coincides with the projective tensor product
$\mathscr{A}\otimes_{\mathrm{max}}\mathscr{A}$
(by Lemma~4.18 of~\cite{Takesaki1}),
which has a universal property similar to that
of the categorical tensor product of von Neumann algebras
(see Proposition 4.7 of~\cite{Takesaki1})
by which one can define a MIU-map
$\delta\colon \mathscr{A}\otimes \mathscr{A}\to\mathscr{A}$
with $\delta(a\otimes b)=ab$ for all~$a,b\in \mathscr{A}$.
Conversely,
if~$(\mathscr{A},m,u)$
is a monoid in~$(\CMIU,\otimes,\mathbb{C})$,
then~$u$ is the unique map~$\mathbb{C}\to\mathscr{A}$,
and by
an easy variation on Lemma~\ref{lem:uniqueness-duplicator} 
for unital $C^*$-algebras,
we see that~$\mathscr{A}$ is commutative, and that~$m(a\otimes b)=ab$
for all~$a,b\in\mathscr{A}$.
Thus the category of (commutative) monoids
on~$(\CMIU,\otimes,\mathbb{C})$
is equivalent to~$\CCMIU$.
Further,
using the obvious strong monoidal adjunction
\[
\xymatrix@C+1pc{
(\mathbf{CH},\times,1)
\ar@/^1.5ex/[r]^-{C}
\ar@{}[r]|-{\bot}
&
\ar@/^1.5ex/[l]^-{\mathrm{sp}}
((\CMIU)^{\op},\otimes,\C)
},
\]
where~$\mathbf{CH}$
is the category of compact Hausdorff spaces
and continuous maps,
one can show (by the same reasoning
as in Theorem~\ref{thm:free-monoid-in-vNAMIU}),
that~$C(\mathrm{sp}(\mathscr{A}))$
is the free (commutative) monoid on~$\CMIU$.

\subsection{Relation to cloning and broadcasting of states}
\label{S:broadcasting}

Our work is motivated by the study of
a model of the quantum lambda calculus,
which contains a duplicability operator ``$\oc$''
based on linear logic.
The need of the linear type system in quantum programming
comes from the well-known fact that quantum states cannot be
cloned~\cite{Dieks1982,WoottersZ1982} nor
broadcast~\cite{BarnumCFJS1996}.
Recently (no-)cloning and (no-)broadcasting
have been studied in categorical quantum mechanics~\cite{Abramsky2010,CoeckeK2015},
and also in operator algebras~\cite{KaniowskiL2015}.

The notion of cloning and broadcasting can be formulated
in von Neumann algebras as follows (the definition is much the same
as \cite{KaniowskiL2015}).
Recall that a normal state of a von Neumann algebra $\scrA$
is a normal (completely) positive unital map $\omega\colon\scrA\to\C$.

\begin{definition}
Let $\scrA$ be a von Neumann algebra.
\begin{enumerate}
\item
A normal state $\omega\colon\scrA\to\C$
is \textbf{cloned} by a normal CPsU-map
$\delta\colon\scrA\otimes\scrA\to\scrA$
if $\omega\circ \delta=\omega\otimes\omega$,
or equivalently, $\omega(\delta(a\otimes b))=
\omega(a)\omega(b)$ for all $a,b\in\scrA$.
\item
A normal state $\omega\colon\scrA\to\C$
is \textbf{broadcast} by a normal CPsU-map
$\delta\colon\scrA\otimes\scrA\to\scrA$
if $\omega(\delta(1\otimes a))=\omega(a)=\omega(\delta(a\otimes1))$
for all $a\in\scrA$.
\end{enumerate}
\end{definition}

There is a clear relationship between duplicability and broadcasting.

\begin{proposition}
A von Neumann algebra $\scrA$ is duplicable if and only if
there exists a normal CPsU-map
$\delta\colon\scrA\otimes\scrA\to\scrA$ 
which broadcasts all normal states of~$\scrA$.
\end{proposition}
\begin{proof}
Suppose that $\scrA$ is duplicable via $\delta\colon \scrA\otimes\scrA\to\scrA$
and $u\in\scrA$.
By Theorem~\ref{thm:duplicable}
and Proposition~\ref{prop:dup-vna-is-monoid},
$\delta$ is a normal CPsU-map and $u=1$.
Then clearly all normal states are broadcast by $\delta$.

Conversely, suppose that
there is $\delta\colon\scrA\otimes\scrA\to\scrA$ such that
$\omega(\delta(1\otimes a))=\omega(a)=\omega(\delta(a\otimes1))$
for all $a\in\scrA$,
and all normal states $\omega$.
Then, since the set of normal states are separating, we obtain
$\delta(1\otimes a)=a=\delta(a\otimes1)$
for all~$a\in\scrA$,
and so~$\delta$ is a duplicator.
\end{proof}

Note that cloning is stronger notion than broadcasting:
if $\omega$ is cloned by $\delta$,
then $\omega(\delta(1\otimes a))=\omega(1)\omega(a)=\omega(a)$
and similarly $\omega(\delta(a\otimes 1))=\omega(a)$.
Even in a duplicable von Neumann algebra,
not every normal state can be cloned.
Indeed, it is easy to see that
a normal state $\omega\colon\scrA\to\C$ on a duplicable von Neumann algebra
is cloned by its duplicator $\delta\colon \scrA\otimes\scrA\to\scrA$
if and only if $\omega$ preserves the multiplication, i.e.\ $\omega$ is
a normal MIU-map.
It follows, by $\vNAMIU(\linf(X),\C)\cong X$
(Lemma~\ref{lem:nsp-linf-ssm}),
that the set of normal states of $\linf(X)$
cloned by its duplicator is precisely the set $X$.

\section*{Acknowledgements}

The research leading to these results has received funding from the
European Research Council under the European Union's Seventh Framework
Programme (FP7/2007-2013) / ERC grant agreement n\textsuperscript{o} 320571.

\bibliographystyle{eptcs}
\bibliography{common}

\begin{thebibliography}{10}
\providecommand{\bibitemdeclare}[2]{}
\providecommand{\surnamestart}{}
\providecommand{\surnameend}{}
\providecommand{\urlprefix}{Available at }
\providecommand{\url}[1]{\texttt{#1}}
\providecommand{\href}[2]{\texttt{#2}}
\providecommand{\urlalt}[2]{\href{#1}{#2}}
\providecommand{\doi}[1]{doi:\urlalt{http://dx.doi.org/#1}{#1}}
\providecommand{\bibinfo}[2]{#2}

\bibitemdeclare{incollection}{Abramsky2010}
\bibitem{Abramsky2010}
\bibinfo{author}{Samson \surnamestart Abramsky\surnameend}
  (\bibinfo{year}{2009}): \emph{\bibinfo{title}{No-Cloning in Categorical
  Quantum Mechanics}}.
\newblock In: {\sl \bibinfo{booktitle}{Semantic Techniques in Quantum
  Computation}}, \bibinfo{publisher}{Cambridge University Press}, pp.
  \bibinfo{pages}{1--28}, \doi{10.1017/CBO9781139193313.002}.

\bibitemdeclare{article}{BarnumCFJS1996}
\bibitem{BarnumCFJS1996}
\bibinfo{author}{Howard \surnamestart Barnum\surnameend},
  \bibinfo{author}{Carlton~M. \surnamestart Caves\surnameend},
  \bibinfo{author}{Christopher~A. \surnamestart Fuchs\surnameend},
  \bibinfo{author}{Richard \surnamestart Jozsa\surnameend} \&
  \bibinfo{author}{Benjamin \surnamestart Schumacher\surnameend}
  (\bibinfo{year}{1996}): \emph{\bibinfo{title}{Noncommuting Mixed States
  Cannot Be Broadcast}}.
\newblock {\sl \bibinfo{journal}{Phys. Rev. Lett.}}
  \bibinfo{volume}{76}(\bibinfo{number}{15}), pp. \bibinfo{pages}{2818--2821},
  \doi{10.1103/PhysRevLett.76.2818}.

\bibitemdeclare{inproceedings}{Benton1995}
\bibitem{Benton1995}
\bibinfo{author}{P.~N. \surnamestart Benton\surnameend} (\bibinfo{year}{1995}):
  \emph{\bibinfo{title}{A Mixed Linear and Non-Linear Logic: Proofs, Terms and
  Models}}.
\newblock In: {\sl \bibinfo{booktitle}{CSL '94}}, {\sl \bibinfo{series}{LNCS}}
  \bibinfo{volume}{933}, \bibinfo{publisher}{Springer}, pp.
  \bibinfo{pages}{121--135}, \doi{10.1007/BFb0022251}.

\bibitemdeclare{article}{cho2014}
\bibitem{cho2014}
\bibinfo{author}{Kenta \surnamestart Cho\surnameend} (\bibinfo{year}{2016}):
  \emph{\bibinfo{title}{Semantics for a Quantum Programming Language by
  Operator Algebras}}.
\newblock {\sl \bibinfo{journal}{New Generation Computing}}
  \bibinfo{volume}{34}(\bibinfo{number}{1}), pp. \bibinfo{pages}{25--68},
  \doi{10.1007/s00354-016-0204-3}.

\bibitemdeclare{article}{CW2016}
\bibitem{CW2016}
\bibinfo{author}{Kenta \surnamestart Cho\surnameend} \&
  \bibinfo{author}{Abraham \surnamestart Westerbaan\surnameend}
  (\bibinfo{year}{2016}): \emph{\bibinfo{title}{Von Neumann Algebras form a
  Model for the Quantum Lambda Calculus}}.
\newblock {\sl \bibinfo{journal}{arXiv preprint arXiv:1603.02133v1 [cs.LO]}}.

\bibitemdeclare{article}{CoeckeK2015}
\bibitem{CoeckeK2015}
\bibinfo{author}{Bob \surnamestart Coecke\surnameend} \& \bibinfo{author}{Aleks
  \surnamestart Kissinger\surnameend} (\bibinfo{year}{2015}):
  \emph{\bibinfo{title}{Categorical Quantum Mechanics {I}: Causal Quantum
  Processes}}.
\newblock {\sl \bibinfo{journal}{arXiv preprint arXiv:1510.05468v2
  [quant-ph]}}.

\bibitemdeclare{book}{conway2007}
\bibitem{conway2007}
\bibinfo{author}{John~B. \surnamestart Conway\surnameend}
  (\bibinfo{year}{1990}): \emph{\bibinfo{title}{A Course in Functional
  Analysis}}, \bibinfo{edition}{second} edition.
\newblock \bibinfo{publisher}{Springer}.

\bibitemdeclare{article}{dauns1972}
\bibitem{dauns1972}
\bibinfo{author}{John \surnamestart Dauns\surnameend} (\bibinfo{year}{1972}):
  \emph{\bibinfo{title}{Categorical $W^*$-tensor product}}.
\newblock {\sl \bibinfo{journal}{Transactions of the American Mathematical
  Society}} \bibinfo{volume}{166}, pp. \bibinfo{pages}{439--456},
  \doi{10.1090/S0002-9947-1972-0295093-6}.

\bibitemdeclare{article}{Dieks1982}
\bibitem{Dieks1982}
\bibinfo{author}{D.~\surnamestart Dieks\surnameend} (\bibinfo{year}{1982}):
  \emph{\bibinfo{title}{Communication by EPR devices}}.
\newblock {\sl \bibinfo{journal}{Physics Letters A}}
  \bibinfo{volume}{92}(\bibinfo{number}{6}), pp. \bibinfo{pages}{271--272},
  \doi{10.1016/0375-9601(82)90084-6}.

\bibitemdeclare{book}{Fremlin2000}
\bibitem{Fremlin2000}
\bibinfo{author}{David~H. \surnamestart Fremlin\surnameend}
  (\bibinfo{year}{2000}): \emph{\bibinfo{title}{Measure Theory}}.
\newblock \bibinfo{publisher}{Torres Fremlin}.

\bibitemdeclare{article}{ndlmcs}
\bibitem{ndlmcs}
\bibinfo{author}{Bart \surnamestart Jacobs\surnameend} (\bibinfo{year}{2015}):
  \emph{\bibinfo{title}{New Directions in Categorical Logic, for Classical,
  Probabilistic and Quantum Logic}}.
\newblock {\sl \bibinfo{journal}{Logical Methods in Computer Science}}
  \bibinfo{volume}{11}(\bibinfo{number}{3}), \doi{10.2168/LMCS-11(3:24)2015}.

\bibitemdeclare{book}{kadison1997}
\bibitem{kadison1997}
\bibinfo{author}{Richard~V. \surnamestart Kadison\surnameend} \&
  \bibinfo{author}{John~R. \surnamestart Ringrose\surnameend}
  (\bibinfo{year}{1997}): \emph{\bibinfo{title}{Fundamentals of the Theory of
  Operator Algebras}}.
\newblock \bibinfo{publisher}{American Mathematical Society}.

\bibitemdeclare{article}{KaniowskiL2015}
\bibitem{KaniowskiL2015}
\bibinfo{author}{Krzysztof \surnamestart Kaniowski\surnameend},
  \bibinfo{author}{Katarzyna \surnamestart Lubnauer\surnameend} \&
  \bibinfo{author}{Andrzej \surnamestart {\L}uczak\surnameend}
  (\bibinfo{year}{2015}): \emph{\bibinfo{title}{Cloning and Broadcasting in
  Operator Algebras}}.
\newblock {\sl \bibinfo{journal}{The Quarterly Journal of Mathematics}}
  \bibinfo{volume}{66}(\bibinfo{number}{1}), pp. \bibinfo{pages}{191--212},
  \doi{10.1093/qmath/hau028}.

\bibitemdeclare{article}{kornell2012}
\bibitem{kornell2012}
\bibinfo{author}{Andre \surnamestart Kornell\surnameend}
  (\bibinfo{year}{2012}): \emph{\bibinfo{title}{Quantum Collections}}.
\newblock {\sl \bibinfo{journal}{arXiv preprint arXiv:1202.2994v1 [math.OA]}}.

\bibitemdeclare{incollection}{Maassen2010}
\bibitem{Maassen2010}
\bibinfo{author}{Hans \surnamestart Maassen\surnameend} (\bibinfo{year}{2010}):
  \emph{\bibinfo{title}{Quantum Probability and Quantum Information Theory}}.
\newblock In: {\sl \bibinfo{booktitle}{Quantum Information, Computation and
  Cryptography}}, {\sl \bibinfo{series}{Lecture Notes in Physics}}
  \bibinfo{volume}{808}, \bibinfo{publisher}{Springer}, pp.
  \bibinfo{pages}{65--108}, \doi{10.1007/978-3-642-11914-9_3}.

\bibitemdeclare{book}{maclane1998}
\bibitem{maclane1998}
\bibinfo{author}{Saunders \surnamestart Mac~Lane\surnameend}
  (\bibinfo{year}{1998}): \emph{\bibinfo{title}{Categories for the Working
  Mathematician}}, \bibinfo{edition}{second} edition.
\newblock \bibinfo{publisher}{Springer}.

\bibitemdeclare{article}{Mellies2009}
\bibitem{Mellies2009}
\bibinfo{author}{Paul-Andr{\'e} \surnamestart Melli{\`e}s\surnameend}
  (\bibinfo{year}{2009}): \emph{\bibinfo{title}{Categorical Semantics of Linear
  Logic}}.
\newblock {\sl \bibinfo{journal}{Panoramas et synth{\`e}ses}}
  \bibinfo{volume}{27}, pp. \bibinfo{pages}{1--196}.

\bibitemdeclare{inproceedings}{MelliesTT2009}
\bibitem{MelliesTT2009}
\bibinfo{author}{Paul-Andr{\'e} \surnamestart Melli{\`e}s\surnameend},
  \bibinfo{author}{Nicolas \surnamestart Tabareau\surnameend} \&
  \bibinfo{author}{Christine \surnamestart Tasson\surnameend}
  (\bibinfo{year}{2009}): \emph{\bibinfo{title}{An Explicit Formula for the
  Free Exponential Modality of Linear Logic}}.
\newblock In: {\sl \bibinfo{booktitle}{ICALP 2009}}, {\sl
  \bibinfo{series}{LNCS}} \bibinfo{volume}{5556},
  \bibinfo{publisher}{Springer}, pp. \bibinfo{pages}{247--260},
  \doi{10.1007/978-3-642-02930-1_21}.

\bibitemdeclare{book}{paulsen2002}
\bibitem{paulsen2002}
\bibinfo{author}{Vern \surnamestart Paulsen\surnameend} (\bibinfo{year}{2002}):
  \emph{\bibinfo{title}{Completely Bounded Maps and Operator Algebras}}.
\newblock \bibinfo{publisher}{Cambridge University Press}.

\bibitemdeclare{book}{Sakai1998}
\bibitem{Sakai1998}
\bibinfo{author}{Sh{\^o}ichir{\^o} \surnamestart Sakai\surnameend}
  (\bibinfo{year}{1998}): \emph{\bibinfo{title}{$C^*$-Algebras and
  $W^*$-Algebras}}.
\newblock \bibinfo{publisher}{Springer}.
\newblock \bibinfo{note}{Reprint of the 1971 edition}.

\bibitemdeclare{incollection}{SelingerV2009}
\bibitem{SelingerV2009}
\bibinfo{author}{Peter \surnamestart Selinger\surnameend} \&
  \bibinfo{author}{Beno{\^\i}t \surnamestart Valiron\surnameend}
  (\bibinfo{year}{2009}): \emph{\bibinfo{title}{Quantum Lambda Calculus}}.
\newblock In: {\sl \bibinfo{booktitle}{Semantic Techniques in Quantum
  Computation}}, \bibinfo{publisher}{Cambridge University Press}, pp.
  \bibinfo{pages}{135--172}, \doi{10.1017/CBO9781139193313.005}.

\bibitemdeclare{book}{Takesaki1}
\bibitem{Takesaki1}
\bibinfo{author}{Masamichi \surnamestart Takesaki\surnameend}
  (\bibinfo{year}{2002}): \emph{\bibinfo{title}{Theory of Operator Algebras
  {I}}}.
\newblock \bibinfo{publisher}{Springer}.
\newblock \bibinfo{note}{2nd printing of the first edition 1979}.

\bibitemdeclare{article}{bram2014}
\bibitem{bram2014}
\bibinfo{author}{Abraham \surnamestart Westerbaan\surnameend}
  (\bibinfo{year}{2015}): \emph{\bibinfo{title}{Quantum Programs as {Kleisli}
  Maps}}.
\newblock {\sl \bibinfo{journal}{arXiv preprint arXiv:1501.01020v2 [math.CT]}}.

\bibitemdeclare{article}{WoottersZ1982}
\bibitem{WoottersZ1982}
\bibinfo{author}{W.~K. \surnamestart Wootters\surnameend} \&
  \bibinfo{author}{W.~H. \surnamestart Zurek\surnameend}
  (\bibinfo{year}{1982}): \emph{\bibinfo{title}{A single quantum cannot be
  cloned}}.
\newblock {\sl \bibinfo{journal}{Nature}} \bibinfo{volume}{299}, pp.
  \bibinfo{pages}{802--803}, \doi{10.1038/299802a0}.

\end{thebibliography}

\end{document}